\documentclass{amsart}
\usepackage{amssymb}
\usepackage[colorlinks = true]{hyperref}

\newcommand{\red}{{\mathrm{red}}}
\newcommand{\unip}{{\mathrm{unip}}}

\DeclareMathOperator{\Aut}{Aut}

\DeclareMathOperator{\Ind}{Ind}
\DeclareMathOperator{\Res}{Res}
\DeclareMathOperator{\Lie}{Lie}
\DeclareMathOperator{\Ext}{Ext}

\newcommand{\idiag}{\overset{\text{\ref{it:diag}}}{\hookrightarrow}}
\newcommand{\ilevi}{\overset{\text{\ref{it:levi}}}{\hookrightarrow}}
\newcommand{\iauto}{\overset{\text{\ref{it:auto}}}{\hookrightarrow}}
\newcommand{\iclass}{\overset{\text{\ref{it:class}}}{\hookrightarrow}}
\newcommand{\imax}{\overset{\text{\ref{it:max}}}{\hookrightarrow}}
\newcommand{\ires}{\overset{\text{\ref{it:resirr}}}{\hookrightarrow}}
\newcommand{\itens}{\overset{\text{\ref{it:tens}}}{\hookrightarrow}}

\newcommand{\bk}{\Bbbk}

\newcommand{\rA}{\mathrm{A}}
\newcommand{\rB}{\mathrm{B}}
\newcommand{\rC}{\mathrm{C}}
\newcommand{\rD}{\mathrm{D}}
\newcommand{\rE}{\mathrm{E}}
\newcommand{\rF}{\mathrm{F}}
\newcommand{\rG}{\mathrm{G}}
\newcommand{\rT}{\mathrm{T}}

\newcommand{\cN}{\mathcal{N}}
\newcommand{\bX}{\mathbf{X}}

\newcommand{\SL}{\mathrm{SL}}
\newcommand{\PSL}{\mathrm{PSL}}
\newcommand{\GL}{\mathrm{GL}}
\newcommand{\SO}{\mathrm{SO}}
\newcommand{\Sp}{\mathrm{Sp}}
\newcommand{\Spin}{\mathrm{Spin}}

\newcommand{\Gm}{\mathbb{G}_m}

\numberwithin{equation}{section}
\newtheorem{thm}{Theorem}[section]
\newtheorem{lem}[thm]{Lemma}
\newtheorem{prop}[thm]{Proposition}
\newtheorem{cor}[thm]{Corollary}

\theoremstyle{definition}

\theoremstyle{remark}
\newtheorem{rmk}[thm]{Remark}

\title{Nilpotent centralizers and good filtrations}

 \author{Pramod N. Achar}
 \address{Department of Mathematics\\
   Louisiana State University\\
   Baton Rouge, LA 70803\\
   U.S.A.}
 \email{pramod@math.lsu.edu}

 \author{William Hardesty}
 \address{School of Mathematics and Statistics\\
   University of Sydney\\
   Camperdown, NSW 2006\\
   U.S.A.}
 \email{hardes11@gmail.com}
 
 \thanks{P.A. was supported by NSF Grant No.~DMS-1802241. W.H. was supported by the ARC Discovery Grant No.~DP170104318.}

\begin{document}

\begin{abstract}
Let $G$ be a connected reductive group over an algebraically closed field $\bk$.  Under mild restrictions on the characteristic of $\bk$, we show that any $G$-module with a good filtration also has a good filtration as a module for the reductive part of the centralizer of a nilpotent element $x$ in its Lie algebra.  
\end{abstract}

\maketitle

\section{Introduction}

Let $G$ be a connected reductive group over an algebraically closed field $\bk$ of characteristic $p > 0$, and let $H$ be a connected reductive subgroup.  Recall that $(G,H)$ is said to be a \emph{Donkin pair} or a \emph{good filtration pair} if every $G$-module with a good filtration still has a good filtration when regarded as an $H$-module.  

Now let $x$ be a nilpotent element in the Lie algebra of $G$, and let $G^x \subset G$ be its stabilizer.  If $p$ is good for $G$, then the theory of \emph{associated cocharacters} is available, and this gives rise to a decomposition
\[
G^x = G^x_\red \ltimes G^x_\unip
\]
where $G^x_\unip$ is a connected unipotent group, and $G^x_\red$ is a (possibly disconnected) group whose identity component $(G^x_\red)^\circ$ is reductive (cf. \cite[5.10]{jan:norp}).  The main result of this paper is the following.

\begin{thm}\label{thm:main}
Let $G$ be a connected reductive group over an algebraically closed field $\bk$ of good characteristic.  For any nilpotent element $x$ in its Lie algebra, $(G, (G^x_\red)^\circ)$ is a Donkin pair.
\end{thm}

Now suppose that $H \subset G$ is a possibly \emph{disconnected} reductive subgroup, i.e., a group whose identity component $H^\circ$ is reductive. If the characteristic of $\bk$ does not divide the order of the finite group $H/H^\circ$, then the category of finite-dimensional $H$-modules is a highest-weight category, as shown in~\cite{ahr}.  In particular, it makes sense to speak of good filtrations for $H$-modules, and so the definition of ``Donkin pair'' makes sense for $(G,H)$.

In order to apply this notion in the case where $H = G^x_\red$, we must impose a slightly stronger condition on $p$: we require it to be \emph{pretty good} in the sense of~\cite[Definition~2.11]{herpel}.  (In general, this condition is intermediate between ``good'' and ``very good.'' It coincides with ``very good'' for semisimple simply-connected groups, whereas for $\GL_n$, all primes are pretty good.)  This is equivalent to requiring $G$ to be \emph{standard} in the sense of~\cite[\S 4]{mcninch-testerman}.  It follows from~\cite[Theorem~1.8]{hardesty} and~\cite[Lemma~2.1]{ahjr} that when $p$ is pretty good for $G$, it does not divide the order of $G^x/(G^x)^\circ \cong G^x_\red/(G^x_\red)^\circ$ for any nilpotent element $x$. As an immediate consequence of Theorem~\ref{thm:main} and Lemma~\ref{lem:donkin-disconn} below, we have the following result.

\begin{cor}\label{cor:main}
Let $G$ be a connected reductive group over an algebraically closed field $\bk$ of pretty good characteristic.  For any nilpotent element $x$ in its Lie algebra, $(G, G^x_\red)$ is a Donkin pair.
\end{cor}

This corollary plays a key role in the proof of the Humphreys conjecture~\cite{ah:sccs}.

The paper is organized as follows: Section~\ref{sec:prelim} contains some general lemmas on Donkin pairs, along with a lengthly list of examples (some previously known, and some new).  Section~\ref{sec:main} gives the proof of Theorem~\ref{thm:main}.  The proof consists of a reduction to the quasi-simple case, followed by case-by-case arguments.  

\begin{rmk}
It would, of course, be desirable to have a uniform proof of Theorem~\ref{thm:main} that avoids case-by-case arguments, perhaps using the method of Frobenius splittings.  Thanks to a fundamental result of Mathieu~\cite{mathieu}, Theorem~\ref{thm:main} would come down to showing that the flag variety of $G$ admits a $(G^x_\red)^\circ$-canonical splitting.  According a result of van der Kallen~\cite{vdk}, this geometric condition is equivalent to a certain linear-algebraic condition (called the ``pairing condition'') on the Steinberg modules for $G$ and $(G^x_\red)^\circ$.  Unfortunately, for the moment, the pairing condition for these groups seems to be out of reach.
\end{rmk}

\section{Preliminaries}
\label{sec:prelim}

\subsection{General lemmas on Donkin pairs}

We begin with three easy statements about good filtrations.

\begin{lem}\label{lem:disconn-gf}
Let $H$ be a possibly disconnected reductive group over an algebraically closed field $\bk$.  Assume that the characteristic of $\bk$ does not divide $|H/H^\circ|$.  An $H$-module $M$ has a good filtration if and only if it has a good filtration as an $H^\circ$-module.
\end{lem}
\begin{proof}
According to~\cite[Eq.~(3.3)]{ahr}, any costandard $H$-module regarded as an $H^\circ$-module is a direct sum of costandard $H^\circ$-modules.  Hence, any $H$-module with a good filtration has a good filtration as an $H^\circ$-module.

For the opposite implication, suppose $M$ is an $H$-module that has a good filtration as an $H^\circ$-module.  To show that it has a good filtration as an $H$-module, we must show that $\Ext^1_H({-},M)$ vanishes on standard $H$-modules.  As explained in the proof of~\cite[Lemma~2.18]{ahr}, we have
\[
\Ext^1_H({-},M) \cong (\Ext^1_{H^\circ}({-},M))^{H/H^\circ},
\]
and the right-hand side clearly vanishes on standard $H$-modules (using~\cite[Eq.~(3.3)]{ahr} again).
\end{proof}

\begin{lem}\label{lem:donkin-disconn}
Let $G$ be a connected, reductive group, and let $H \subset G$ be a possibly disconnected reductive subgroup.  Assume that the characteristic of $\bk$ does not divide $|H/H^\circ|$.  Then $(G,H)$ is a Donkin pair if and only if $(G,H^\circ)$ is a Donkin pair.
\end{lem}
\begin{proof}
This is an immediate consequence of Lemma~\ref{lem:disconn-gf}.
\end{proof}

\begin{lem}\label{lem:derived}
Let $G$ be a connected, reductive group, and let $G'$ be its derived subgroup.  Let $H \subset G$ be a connected, reductive subgroup.  Then $(G,H)$ is a Donkin pair if and only if $(G',(G' \cap H)^\circ)$ is a Donkin pair.
\end{lem}
\begin{proof}
Let $T \subset B \subset G$ denote a maximal torus and Borel subgroup respectively, and suppose that $G''$ is any closed connected subgroup satisfying 
$G' \subseteq G'' \subseteq G$. Now let $T'' = G''\cap T$, $B'' = G'' \cap B$, and observe that by \cite[I.6.14(1)]{jan:rag}, we have
\begin{equation}\label{eqn:ind-rest-comm}
\Res^G_{G''}\Ind_B^GM \cong \Ind_{B''}^{G''}\Res^B_{B''}M
\end{equation}
 for any $B$-module $M$. Thus, for any dominant weight $\lambda \in \bX(T)^+$ where we set
 $\lambda'' = \Res^T_{T''}(\lambda) \in \bX(T'')^+$, it follows that 
 $\Res^G_{G''}\Ind_B^G(\lambda) \cong \Ind_{B''}^{G''}(\lambda'')$.
Thus $(G,G'')$ is always a Donkin pair. 
Furthermore, if we let $H' \subseteq H$ be the derived subgroup and $H'' = (G'\cap H)^\circ$, 
then  $H' \subseteq H'' \subseteq H$.  We can therefore apply \cite[I.6.14(1)]{jan:rag} again to show that 
$(H,H'')$ is a Donkin pair. 

Now suppose $(G, H)$ is a Donkin pair. In this case it immediately follows from above that $(G, H'')$ is a 
Donkin pair. Moreover, if we let $T' = G'\cap T$ and $B' = G'\cap B$, then \eqref{eqn:ind-rest-comm} actually implies that for any $\lambda' \in \bX(T')^+$, there exists $\lambda \in \bX(T)^+$ with $\lambda' = \Res^T_{T'}(\lambda)$ such that
\[
 \Res^G_{G'}\Ind_B^G(\lambda) \cong \Ind_{B'}^{G'}(\lambda').
\]
In particular, 
\[
 \Res^{G'}_{H''}\Ind_{B'}^{G'}(\lambda') \cong  \Res^{G}_{H''}\Ind_{B}^{G}(\lambda)
\]
has a good filtration as an $H''$-module, and hence, $(G', H'')$ is also a Donkin pair. 

Conversely, suppose that $(G', H'')$ is a Donkin pair. We can first deduce that 
$(G, H'')$ is a Donkin pair from the fact that $(G,G')$ is a Donkin pair. Also, by similar arguments as above we
can see that for any $\mu'' \in \bX(H''\cap T)^+$, there exists some $\mu \in \bX(H\cap T)^+$ with
$
\mu'' = \Res^{H\cap T}_{H''\cap T}(\mu),
$
such that 
\[
 \Res^H_{H''}\Ind_{H\cap B}^H(\mu) \cong \Ind_{H''\cap B}^{H''}(\mu'').
\]
This implies that an $H$-module $M$ has a good filtration if and only if the $H''$-module $\Res^H_{H''}M$ has 
a good filtration. Therefore, $(G,H)$ is also a Donkin pair. 
\end{proof}

\subsection{Examples of Donkin pairs}

The following proposition collects a number of known examples of Donkin pairs.  The last five parts of the proposition deal with various examples where $G$ is quasi-simple and simply connected.  For pairs of the form $(\Spin_n, H)$, it is usually more convenient to describe the image $H'$ of $H$ under the map $\pi: \Spin_n \to \SO_n$.  Of course, $H$ can be recovered from $H'$, as the identity component of $\pi^{-1}(H)$.  We use the notation that
\[
(\SO_n,H')^\sim = (\Spin_n, H).
\]
It should be noted that the following proposition does not exhaust the known examples in the literature: for instance, according to~\cite{brundan}, there is a Donkin pair of type $(\rB_3, \rG_2)$, but this example is not needed in the present paper.

\begin{prop}\label{prop:pairlist}
Let $G$ be a connected, reductive group, and let $H \subset G$ be a closed, connected, reductive subgroup.  If the pair $(G,H)$ satisfies one of the following conditions, then it is a Donkin pair.
\begin{enumerate}
\item $G = H \times \cdots \times H$, and $H \hookrightarrow G$ is the diagonal embedding.\label{it:diag}
\item $H$ is a Levi subgroup of $G$.\label{it:levi}
\end{enumerate}
For the remaining parts, assume that $G$ is quasi-simple and simply connected.
\begin{enumerate}
\setcounter{enumi}{2}
\item $G$ is of simply-laced type, and $H$ is the fixed-point set of a diagram automorphism of $G$:\label{it:auto}
\begin{align*}
(\rA_{2n-1}, \rC_n) &= (\SL_{2n},\Sp_{2n}) & (\rD_4,\rG_2)&=(\Spin_8, \rG_2) \\
(\rD_n, \rB_{n-1}) &= (\SO_{2n}, \SO_{2n-1})^\sim & &( \rE_6, \rF_4)
\end{align*}
\item Certain embeddings of classical groups:\label{it:class}
\begin{align*}
\left.
\begin{array}{c}
(\rA_{2n}, \rB_n) \\
(\rA_{2n-1}, \rD_n)
\end{array}
\right\}
&= (\SL_r, \SO_r) \qquad (p > 2) \\
(\rA_{2n-1}, \rC_n)
&= (\SL_{2n}, \Sp_{2n}) \\
\left.
\begin{array}{c}
(\rB_{n+m}, \rB_n\rD_m) \\
(\rD_{n+m}, \rD_n\rD_m) \\
(\rD_{n+m+1}, \rB_n\rB_m) 
\end{array}
\right\} &= (\SO_{r+s}, \SO_r \times \SO_s)^\sim \quad(p > 2)
\\
(\rC_{n+m}, \rC_n\rC_m) &= (\Sp_{2n+2m}, \Sp_{2n} \times \Sp_{2m})
\end{align*}
\item Certain maximal-rank subgroups of exceptional groups:\label{it:max}
\begin{align*}
&&
   (\rE_8 &, \rA_2\rE_6)\quad (p > 5) \\
(\rE_8 &, \rD_8)\quad (p > 2) &
   (\rE_8 &, \rA_1\rA_2\rA_5)\quad (p > 5) \\
(\rE_8 &, \rA_1\rE_7)\quad (p > 2 )  &
   (\rE_8 &, \rA_3\rD_5)\quad (p > 5) \\
(\rE_7 &, \rA_1\rD_6)\quad (p > 2) &
   (\rE_8 &, \rA_4\rA_4)\quad (p > 5) \\
(\rF_4 &, \rB_4) \quad (p > 2) &
   (\rF_4 &, \rA_3\rA_1)\quad (p > 3) \\
&&
   (\rG_2 &, \rA_1\rA_1)
\end{align*}
\item Certain restricted irreducible representations:\label{it:resirr}
\begin{align*}
(\rA_n &, \rA_1) \quad (p > n) \\
(\rA_7 &, \rA_2) \quad (p > 3) \\
(\rA_6 &, \rG_2) \quad (p > 3)
\end{align*}
\item Tensor product embeddings of classical groups ($p > 2$):\label{it:tens}
\begin{align*}
\left.
\begin{array}{c}
(\rC_{(2n+1)m}, \rB_n)\\
(\rC_{2nm} , \rD_n)
\end{array}
\right\}
&= (\Sp_{2rm}, \SO_r)
&
\left.
\begin{array}{c}
(\rB_{n+m+2nm}, \rB_n) \\
(\rD_{(2n+1)m}, \rB_n)\\
(\rD_{nm} , \rD_n)
\end{array}
\right\}
&= (\SO_{rs}, \SO_r)^\sim
\\
(\rD_{2nm}, \rC_n) &= (\SO_{4nm}, \Sp_{2n})^\sim
&
(\rC_{nm}, \rC_n) &= (\Sp_{2nm}, \Sp_{2n})
\end{align*}
\end{enumerate}
\end{prop}

The details of the embeddings in parts~\eqref{it:resirr} and~\eqref{it:tens} will be described below.  

\begin{proof}[Proofs for parts~\eqref{it:levi}--\eqref{it:max}]
Parts~\eqref{it:diag} and~\eqref{it:levi} are due to Mathieu~\cite{mathieu} (following earlier work of Donkin~\cite{donkin} that covered most cases).  Parts~\eqref{it:auto} and~\eqref{it:class}, with the exception of the pair $(\rE_6,\rF_4)$, are due to Brundan~\cite{brundan}.  The pair $(\rG_2,\rA_1\rA_1)$ in part~\eqref{it:max} is also due to Brundan~\cite{brundan}.  The pair $(\rE_6,\rF_4)$ and the pairs in the first column of part~\eqref{it:max} are due to van der Kallen~\cite{vdk}.  The pairs in the second column of part~\eqref{it:max} are due to Hague--McNinch~\cite{hm}.
\end{proof}

\begin{proof}[Proof of part~\eqref{it:resirr}]
Each pair $(\rA_n,H) = (\SL_{n+1},H)$ in this statement arises from some $(n+1)$-dimensional representation of $H$.  Call that representation $V$.   The representations $V$ are as follows:
\begin{itemize}
\item $(\rA_n,\rA_1)$: the dual Weyl module for $\SL_2$ of highest weight $n$
\item $(\rA_7,\rA_2)$: the adjoint representation of $\PSL_3$
\item $(\rA_6,\rG_2)$: the $7$-dimensional dual Weyl module whose highest weight is the short dominant root
\end{itemize}
According to~\cite[Lemma~3.2(iv)]{brundan} or~\cite[\S 3.2.6]{hm}, to prove the claim, we must show that each exterior algebra $\bigwedge^\bullet V$ has a good filtration as an $H$-module.  For $(\rA_n,\rA_1)$, this is shown in~\cite[\S 3.4.3]{hm}.  For $(\rA_7, \rA_2)$ and $(\rA_6, \rG_2)$, explicit calculations using the LiE software package~\cite{lie} show that the character of $\bigwedge^\bullet V$ is the sum of characters of dual Weyl modules whose highest weights are restricted weights when $p > 3$.
\end{proof}

\begin{proof}[Proof of part~\eqref{it:tens}]
To define the group embeddings in this statement, we will assume that $G$ is either $\Sp_{2n}$ or $\SO_n$.  However, in the latter case, the proof that $(G,H)$ is a Donkin pair will also imply the corresponding statement for $G = \Spin_n$.

Let $V_1$ be a vector space equipped with a nondegenerate bilinear form $B_1$ satisfying $B_1(v,w) = \varepsilon_1 B_1(w,v)$, where $\varepsilon_1 = \pm 1$, and let $\Aut(V_1,B_1)^\circ$ be the connected group of linear automorphisms of $V_1$ that preserve $B_1$.  This group is either $\SO_{\dim V}$ or $\Sp_{\dim V}$, depending on $\varepsilon_1$.  Let $V_2$, $B_2$, $\varepsilon_2$ be another collection of similar data.  Then $B_1 \otimes B_2$ is a nondegenerate pairing on $V_1 \otimes V_2$, with sign $\varepsilon_1\varepsilon_2$.  We obtain an embedding
\[
\Aut(V_1,B_1)^\circ \times \Aut(V_2,B_2)^\circ \hookrightarrow \Aut(V_1 \otimes V_2, B_1 \otimes B_2)^\circ.
\]
Now restrict to just one factor:
\begin{equation}\label{eqn:tensor-defn}
\Aut(V_1,B_1)^\circ \hookrightarrow \Aut(V_1 \otimes V_2, B_1 \otimes B_2)^\circ.
\end{equation}
The four kinds of pairs listed in the statement are all instances of this embedding, depending on the signs $\varepsilon_1$ and $\varepsilon_2$.  We will now prove that
\[
(G,H) = (\Aut(V_1 \otimes V_2, B_1 \otimes B_2)^\circ, \Aut(V_1,B_1)^\circ)
\]
is a Donkin pair.  Let $r = \dim V_1$ and $s = \dim V_2$.

Suppose first that $\varepsilon_2 = 1$.  Then~\eqref{eqn:tensor-defn} corresponds to either $(\SO_{rs}, \SO_r)$ or $(\Sp_{rs}, \Sp_{r})$.  In this case, $V_2$ admits an orthonormal basis $x_1, \ldots, x_s$, where
\[
B_2(x_i,x_j) = \delta_{ij}.
\]
Then the group $\Aut(V_1,B_1)^\circ$ preserves each $V_1 \otimes x_i \subset V_1 \otimes V_2$.  In this case, the embedding~\eqref{eqn:tensor-defn} can be factored as
\[
\Aut(V_1,B_1)^\circ \idiag \underbrace{\Aut(V_1,B_1)^\circ \times \cdots \times \Aut(V_1,B_1)^\circ}_{\text{$s$ copies}} \iclass  \Aut(V_1 \otimes V_2, B_1 \otimes B_2)^\circ.
\]
The first map is a diagonal embedding; it results in a Donkin pair by part~\eqref{it:diag} of the proposition.  The second embedding gives a Donkin pair by part~\eqref{it:class}.

Next, suppose that $\varepsilon_2 = -1$, and assume for now that $\dim V_2 = 2$.  Choose a basis $\{x,y\}$ for $V_2$ such that $B_2(x,y) = 1$.  Then $V_1 \otimes x$ and $V_1 \otimes y$ are both maximal isotropic subspaces of $V_1 \otimes V_2$.  Define an action of $\GL(V_1)$ on $V_1 \otimes V_2$ as follows:
\[
\begin{aligned}
g \cdot (v \otimes x) &= (gv) \otimes x, \\
g \cdot (v \otimes y) &= ((g^{\mathrm{t}})^{-1}v) \otimes y
\end{aligned}
\qquad
\text{for $g \in \GL(V_1)$,}
\]
where $g^{\mathrm{t}}$ denotes the adjoint operator to $g$ with respect to the nondegenerate form on $V_1$.  This action defines an embedding of $\GL(V_1)$ in $\Aut(V_1 \otimes V_2, B_1 \otimes B_2)^\circ$.  In fact, it identifies $\GL(V_1)$ with a Levi subgroup of $\Aut(V_1 \otimes V_2, B_1 \otimes B_2)^\circ$.  (This is the usual embedding of $\GL_r$ as a Levi subgroup in either $\SO_{2r}$ or $\Sp_{2r}$.)  The embedding~\eqref{eqn:tensor-defn} then factors as
\[
\Aut(V_1,B_1)^\circ \iclass \GL(V_1) \ilevi \Aut(V_1 \otimes V_2, B_1 \otimes B_2)^\circ.
\]
The first embedding gives a Donkin pair by part~\eqref{it:class} of the proposition, and the second by part~\eqref{it:levi}.

Finally, suppose $\varepsilon_2 = -1$ and $s = \dim V_2 > 2$.  This dimension must still be even, say $s = 2m$.  Choose a basis $x_1, \ldots, x_m, y_1, \ldots, y_m$ for $V_2$ such that
\[
B_2(x_i,x_j) = B_2(y_i,y_j) = 0,
\qquad
B_2(x_i,y_j) = \delta_{ij}.
\]
Let $V_2^{(i)}$ be the $2$-dimensional subspace spanned by $x_i$ and $y_i$.  Then $B_2$ restricts to a nondegenerate symplectic form $B_2^{(i)}$ on $V_2^{(i)}$.  We factor~\eqref{eqn:tensor-defn} as follows:
\begin{multline*}
\Aut(V_1,B_1)^\circ \idiag
\underbrace{\Aut(V_1,B_1)^\circ \times \cdots \times \Aut(V_1,B_1)^\circ}_{\text{$m$ copies}} \\
\hookrightarrow \Aut(V_1 \otimes V_2^{(1)},B_1 \otimes B_2^{(1)})^\circ \times \cdots \times \Aut(V_1 \otimes V_2^{(m)},B_1 \otimes B_2^{(m)})^\circ \\
\iclass \Aut(V_1 \otimes V_2, B_1 \otimes B_2)^\circ.
\end{multline*}
Here, the first arrow is a diagonal embedding (part~\eqref{it:diag} of the proposition); the second arrow is several instances of the embedding from the previous paragraph (since $\dim V_2^{(i)} = 2$); and the last arrow comes from part~\eqref{it:class} of the proposition.  We thus again obtain a Donkin pair.
\end{proof}

\section{Proof of Theorem~\ref{thm:main}}
\label{sec:main}

\subsection{Reduction to the quasi-simple case}

Let $G$ be an arbitrary connected reductive group in good characteristic.
For any nilpotent element $x \in \Lie(G)$, there exists a cocharacter 
$\tau: \Gm \rightarrow G$ and a Levi subgroup $L_\tau \subset G$ such that $L_\tau$ is the centralizer of 
the subgroup $\tau(\Gm)$, where
$
G^x_\red = L_\tau \cap G^x$ (cf. \cite[5.10]{jan:norp}).

If we let $G'$ be the derived subgroup of $G$, then any nilpotent element $x$ for $G$ also satisfies $x \in \Lie(G')$, 
and by \cite[5.9]{jan:norp},  $\tau(\Gm) \subset G' \subseteq G$. In particular, $(G')^x = G'\cap G^x$ and
$L'_{\tau} = G'\cap L_{\tau}$
is the centralizer of $\tau(\Gm)$ in $G'$. Thus,
\[
(G')^x_{\red} = G' \cap G^x_\red. 
\] 
It now follows from Lemma~\ref{lem:derived} that $(G,(G^x_\red)^\circ)$ is a Donkin pair if and only if 
$(G', ((G')^x_\red)^\circ)$ is a Donkin pair. 
So we can reduce to the case where $G$ is semi-simple. 

Suppose now that $\pi: G \twoheadrightarrow \bar G$ is an isogeny (i.e.
surjective with finite central kernel), where $G$ is an arbitrary connected reductive group in good characteristic.  Then, by~\cite[Proposition~2.7(a)]{jan:norp}, $\pi$ induces a bijection between the nilpotent elements in $\Lie(G)$ and those in $\Lie(\bar G)$, and for any nilpotent element $x \in \Lie(G)$, we have $\pi(G^x) = \bar G^{\pi(x)}$. Moreover, by similar arguments as above we
can also deduce that $\pi(G^x_\red) = \bar G^{\pi(x)}_\red$ (cf. \cite[5.9]{jan:norp}). In particular, 
\[
\pi((G^x_\red)^\circ) = \pi(G^x_\red)^\circ = (\bar G^{\pi(x)}_\red)^\circ,
\]
since any surjective morphism of algebraic groups takes the identity component to the identity component. 

Let $H = (G^x_\red)^\circ$ and $\bar H = (\bar G^{\pi(x)}_\red)^\circ$, and note that for any $\bar G$-module $M$, there is a 
natural isomorphism 
\[
\Res^G_H\Res^{\bar G}_GM \cong \Res^{\bar H}_{H}\Res^{\bar G}_{\bar H}M. 
\]
From this we can see that if $(G,H)$ is a Donkin pair, then $(\bar G,\bar H)$ must also be a Donkin pair, since it is 
straightforward to check that a
$\bar G$-module $M$ (resp.~an $\bar H$-module $N$) has a good filtration if and only if 
 $\Res^{\bar G}_G M$ (resp.~$\Res^{\bar H}_{H} N$) has a good filtration. This allows us to reduce to the case
 where $G$ is semisimple and simply connected. 

Finally, suppose that that $G = G_1 \times G_2$ where $G_1$, $G_2$ are connected reductive groups in good characteristic. Let 
$x = (x_1,x_2) \in \Lie(G_1)\oplus \Lie(G_2)$ be an arbitrary nilpotent element. We can immediately see that
\[
(G^x_\red)^\circ = ((G_1)^{x_1}_\red)^\circ \times  ((G_2)^{x_2}_\red)^\circ.
\]
It now follows from the general properties of induction for direct products (see \cite[I.3.8]{jan:rag}) that 
$(G,(G^x_\red)^\circ)$ is a Donkin pair if and only if $(G_1,((G_1)^{x_1}_\red)^\circ)$ and 
$(G_2,((G_2)^{x_2}_\red)^\circ)$ are Donkin pairs. 
Therefore, by the well-known fact that any simply connected semisimple group is a 
direct product of quasi-simple simply connected groups, we can reduce the proof of Theorem~\ref{thm:main} to the 
case where $G$ is quasi-simple. 

\subsection{Proof for classical groups}

We now prove the theorem for the groups $\GL_n$, $\Sp_n$, and $\Spin_n$.  For the last case, we will actually describe the group $(G^x_\red)^\circ$ and its embedding in $G$ for $\SO_n$ instead, but the proof of the Donkin pair property will also apply to $\Spin_n$.

Let $x$ be a nilpotent element in the Lie algebra of one of $\GL_n$, $\Sp_n$, or $\SO_n$.  Let $\mathbf{s} = [s_1^{r_1}, s_2^{r_2}, \ldots, s_k^{r_k}]$ be the partition of $n$ that records the sizes of the Jordan blocks of $x$.  (This means that $x$ has $r_1$ Jordan blocks of size $s_1$, and $r_2$ Jordan blocks of size $s_2$, etc.)  The vector space $V = \bk^n$ can be decomposed as
\[
V = V^{(1)} \oplus V^{(2)} \oplus \cdots \oplus V^{(k)}
\]
where each $V^{(i)}$ is preserved by $x$, and $x$ acts on $V^{(i)}$ by Jordan blocks of size $s_i$.  (Thus, $\dim V^{(i)} = r_is_i$.)  When $G$ is $\Sp_n$ or $\SO_n$, the nondegenerate bilinear form on $V$ restricts to a nondegenerate form of the same type on each $V^{(i)}$.

The description of $(G^x_\red)^\circ$ in~\cite[Chapter~3]{ls} shows that it factors through the appropriate embedding below:
\begin{align*}
\GL(V^{(1)}) \times \cdots \times \GL(V^{(k)}) &\hookrightarrow \GL(V) \\
\Sp(V^{(1)}) \times \cdots \times \Sp(V^{(k)}) &\hookrightarrow \Sp(V) \\
\SO(V^{(1)}) \times \cdots \times \SO(V^{(k)}) &\hookrightarrow \SO(V)
\end{align*}
All three of these embeddings give Donkin pairs: in the case of $\GL_n$, it is an inclusion of a Levi subgroup (Proposition~\ref{prop:pairlist}\eqref{it:levi}); and in the case of $\Sp_n$ or $\SO_n$, it falls under Proposition~\ref{prop:pairlist}\eqref{it:class}.

We can therefore reduce to the case where $x$ has Jordan blocks of a single size.  Suppose from now on that $\mathbf{s} = [s^r]$.  Then there exists a vector space isomorphism
\[
V \cong V_1 \otimes V_2
\]
where $\dim V_1 = r$ and $\dim V_2 = s$, and such that $x$ corresponds to $\mathrm{id}_{V_1} \otimes N$, where $N: V_2 \to V_2$ is a nilpotent operator with a single Jordan block (of size $s$).  

Suppose now that $G = \GL(V)$.  Then, according to~\cite[Proposition~3.8]{ls}, we have $(G^x_\red)^\circ \cong \GL(V_1)$. Choose a basis $\{v_1, \ldots, v_s\}$ for $V_2$.  The embedding of $(G^x_\red)^\circ$ in $G$ factors as
\[
\GL(V_1) \hookrightarrow \GL(V_1 \otimes v_1) \times \cdots \times \GL(V_1 \otimes v_s) \hookrightarrow \GL(V).
\]
The first map above is a diagonal embedding (Proposition~\ref{prop:pairlist}\eqref{it:diag}), and the second is the inclusion of a Levi subgroup (Proposition~\ref{prop:pairlist}\eqref{it:levi}), so $(G, (G^x_\red)^\circ)$ is a Donkin pair in this case.

Next, suppose $G = \Sp(V)$ or $\SO(V)$.  According to~\cite[Proposition~3.10]{ls}, both $V_1$ and $V_2$ can be equipped with nondegenerate bilinear forms $B_1$ and $B_2$ such that $B_1 \otimes B_2$ agrees with the given bilinear form on $V$.  Moreover, $(G^x_\red)^\circ = \Aut(V_1,B_1)^\circ$.  We are thus in the setting of Proposition~\ref{prop:pairlist}\eqref{it:tens}. \qed

\begin{table}
\[
\begin{array}{ll}
\text{Orbit} & (G^x_\red)^\circ \\
\hline
\rA_1 & \rE_7 \ilevi \rE_8 \\
\rA_1^2 & \rB_6 \ilevi \rB_1\rB_6 \iclass \rD_8 \imax \rE_8 \\
\rA_2 & \rE_6 \ilevi \rA_2\rE_6 \imax \rE_8 \\
\rA_1^3 & \rA_1 \rF_4 \iauto \rA_1 \rE_6 \ilevi \rE_8 \\
\rA_2\rA_1 & \rA_5 \ilevi \rA_1\rA_2\rA_5 \imax \rE_8 \\
\rA_3 & \rB_5 \ilevi \rB_2\rB_5 \iclass \rD_8 \imax \rE_8 \\
\rA_1^4 & \rC_4 \itens \rD_8 \imax \rE_8 \\
\rA_2\rA_1^2 & \rA_1\rB_3 = \rB_1\rB_3 \idiag \rB_1\rB_1\rB_1\rB_3 \iclass \rD_8 \imax \rE_8 \\
\rA_2^2 & \rG_2\rG_2 \iauto \rD_4\rD_4 \iclass \rD_8 \imax \rE_8 \\
\rA_3\rA_1 & \rA_1\rB_3 = \rC_1 \rB_3 \itens \rD_2\rB_3 \ilevi \rB_2 \rD_2 \rB_3 \iclass \rD_8 \imax \rE_8 \\
\rA_4 & \rA_4 \ilevi \rE_8 \\
\rD_4 & \rF_4 \iauto \rE_6 \ilevi \rE_8 \\
\rD_4(a_1) & \rD_4 \ilevi \rE_8 \\
\rA_2\rA_1^3 & \rA_1 \rG_2 \ires \rA_1\rA_6 \ilevi \rA_1\rE_7 \imax \rE_8 \\
\rA_2^2\rA_1 & \rA_1 \rG_2 \ilevi \rG_2\rG_2 \iauto \rD_4\rD_4 \iclass \rD_8 \imax \rE_8 \\
\rA_3\rA_1^2 & \rA_1\rB_2 = \rB_1\rC_2 \itens \rB_1\rD_4 \ilevi \rB_1\rD_4 \rB_2 \iclass \rD_8 \imax \rE_8 \\
\rA_3\rA_2 & \rB_2\rT_1 = \rB_2\rD_1 \itens \rB_2\rD_3 \ilevi \rB_2\rD_3\rB_2 \iclass \rD_8 \imax \rE_8 \\
\rA_4\rA_1 & \rA_2\rT_1 \ilevi \rA_4\rA_4 \imax \rE_8 \\
\rD_4\rA_1 & \rC_3 \ilevi \rF_4 \iauto \rE_6 \ilevi \rE_8 \\
\rD_4(a_1)\rA_1 & \rA_1^3 \ilevi \rE_8 \\
\rA_5 & \rA_1\rG_2 \iauto \rA_1\rD_4 \ilevi \rE_8 \\
\rD_5 & \rB_3 \ilevi \rB_3\rB_4 \iclass \rD_8 \imax \rE_8 \\
\rD_5(a_1) & \rA_3 \ilevi \rA_3\rD_5 \imax \rE_8 \\
\rA_2^2\rA_1^2 & \rB_2 \itens \rB_7 \iclass \rD_8 \imax \rE_8 \\
\rA_3\rA_2\rA_1 & \rA_1\rA_1 \idiag \rA_1(\rA_1\rA_1) \ires \rA_1(\rA_4\rA_2) \ilevi \rA_1\rE_7 \imax \rE_8 \\
\rA_3^2 & \rC_2 \itens \rD_8 \imax \rE_8 \\
\rA_4\rA_1^2 & \rA_1\rT_1 \idiag \rA_1\rA_1\rT_1 \ilevi \rA_4 \ilevi \rE_8 \\
\rA_4\rA_2 & \rA_1\rA_1 \idiag \rA_1(\rA_1\rA_1\rA_1) \ires \rA_1(\rA_1\rA_2\rA_3) \ilevi \rA_1\rE_7 \imax \rE_8 \\
\rD_4\rA_2 & \tilde{\rA}_2 \ilevi \rF_4 \iauto \rE_6 \ilevi \rE_8 \\
\end{array}
\]
\caption{Nilpotent centralizers in $\rE_8$}
\label{tab:e8-1}
\end{table}

\begin{table}
\[
\begin{array}{ll}
\text{Orbit} & (G^x_\red)^\circ \\
\hline
\rD_4(a_1)\rA_2 & \rA_2 \ires \rA_7 \ilevi \rE_8 \\
\rA_5\rA_1 & \rA_1\rA_1 \ilevi \rA_1\rG_2 \iauto \rA_1\rD_4 \ilevi \rA_1\rE_7 \imax \rE_8 \\
\rD_5\rA_1 & \rA_1\rA_1 = \rB_1\rC_1 \itens \rB_1\rD_2 \ilevi \rB_4\rB_1\rD_2 \iclass \rD_8 \imax \rE_8 \\
\rD_5(a_1)\rA_1 & \rA_1\rA_1 \ires \rA_1 \tilde{\rA_2} \ilevi \rA_1\rF_4 \iauto \rA_1\rE_6 \ilevi \rA_1\rE_7 \imax \rE_8 \\
\rA_6 & \rA_1^2 \ilevi \rA_1\rE_7 \imax \rE_8 \\
\rD_6 & \rB_2 \ilevi \rB_2\rB_5 \iclass \rD_8 \imax \rE_8 \\
\rD_6(a_1) & \rA_1\rA_1 = \rD_2 \ilevi \rD_2 \rD_6 \iclass \rD_8 \imax \rE_8 \\
\rD_6(a_2) & \rA_1\rA_1 = \rD_2 \ilevi \rD_2 \rD_6 \iclass \rD_8 \imax \rE_8 \\
\rE_6 & \rG_2 \iauto \rD_4 \ilevi \rE_8 \\
\rE_6(a_1) & \rA_2 \ilevi \rA_2\rE_6 \imax \rE_8 \\
\rE_6(a_3) & \rG_2  \iauto \rD_4 \ilevi \rE_8 \\
\rA_4\rA_2\rA_1 & \rA_1 \idiag \rA_1\rA_1\rA_1 \ires \rA_1\rA_2\rA_3 \ilevi \rE_7 \ilevi \rE_8 \\
\rA_4\rA_3 & \rA_1 = \rB_1 \itens \rB_7 \iclass \rD_8 \imax \rE_8 \\
\rD_5\rA_2 & \rT_1 \\
\rD_5(a_1)\rA_2 & \rA_1 = \rB_1 \itens \rB_4 \ilevi \rB_4\rB_3 \iclass \rD_8 \imax \rE_8 \\
\rA_6\rA_1 & \rA_1 \ilevi \rE_8 \\
\rE_6\rA_1 & \rA_1 \ilevi \rG_2 \iauto \rD_4 \ilevi \rE_8 \\
\rE_6(a_1)\rA_1 & \rT_1 \\
\rE_6(a_3)\rA_1 & \rA_1 \ilevi \rG_2 \iauto \rD_4 \ilevi \rE_8 \\
\rA_7 & \rA_1 = \rC_1 \itens \rD_8 \imax \rE_8 \\
\rD_7 & \rA_1 = \rB_1 \ilevi \rB_1 \rB_6 \iclass \rD_8 \imax \rE_8 \\
\rD_7(a_1) & \rT_1 \\
\rD_7(a_2) & \rT_1 \\
\rE_7 & \rA_1 \ilevi \rA_1\rE_7 \imax \rE_8 \\
\rE_7(a_1) & \rA_1 \ilevi \rA_1\rE_7 \imax \rE_8 \\
\rE_7(a_2) & \rA_1 \ilevi \rA_1\rE_7 \imax \rE_8 \\
\rE_7(a_3) & \rA_1 \ilevi \rA_1\rE_7 \imax \rE_8 \\
\rE_7(a_4) & \rA_1 \ilevi \rA_1\rE_7 \imax \rE_8 \\
\rE_7(a_5) & \rA_1 \ilevi \rA_1\rE_7 \imax \rE_8
\end{array}
\]
\caption{Nilpotent centralizers in $\rE_8$, continued}
\label{tab:e8-2}
\end{table}

\begin{table}
\[
\begin{array}{ll}
\text{Orbit} & (G^x_\red)^\circ \\
\hline
\rA_1 & \rD_6 \ilevi \rE_7 \\
\rA_1^2 & \rA_1 \rB_4 = \rB_1\rB_4 \iclass \rD_6 \ilevi \rE_7 \\
\rA_2 & \rA_5 \ilevi \rE_7 \\
\rA_2\rA_1 & \rA_3\rT_1 \ilevi \rE_7 \\
(\rA_1^3)' & \rF_4 \iauto \rE_6 \ilevi \rE_7 \\
(\rA_1^3)'' & \rA_1\rC_3  \iauto \rA_1\rA_5 \ilevi \rE_7 \\
\rA_3 & \rA_1\rB_3 = \rB_1\rB_3 \iclass \rD_5 \ilevi \rE_7 \\
\rA_1^4 & \rC_3 \itens \rD_6 \ilevi \rE_7 \\
\rA_2\rA_1^2 & \rA_1\rA_1\rA_1 = \rB_1\rA_1\rB_1 \itens \rB_4\rA_1\rB_1 \iclass \rA_1\rD_6 \imax \rE_7 \\
\rA_2^2 & \rA_1\rG_2 \iauto \rA_1\rD_4 \idiag \rA_1\rA_1\rA_1\rD_4 = \rA_1\rD_2\rD_4 \iclass \rA_1\rD_6 \imax \rE_7 \\
\rA_4 & \rA_2\rT_1 \ilevi \rE_7 \\
\rD_4 & \rC_3 \iauto \rA_5 \ilevi \rE_7 \\
\rD_4(a_1) & \rA_1\rA_1\rA_1 \ilevi \rE_7 \\
\rA_2\rA_1^3 & \rG_2 \ires \rA_6 \ilevi \rE_7 \\
\rA_2^2\rA_1 & \rA_1\rA_1 \idiag (\rA_1\rA_1\rA_1)\rG_2 \iauto \rA_1\rD_2\rD_4 \iclass \rA_1\rD_6 \imax \rE_7 \\
(\rA_3\rA_1)' & \rB_3 \iclass \rD_4 \ilevi \rE_7 \\
(\rA_3\rA_1)'' & \rA_1\rA_1\rA_1 = \rC_1\rA_1\rB_1 \itens \rD_2 \rA_1 \rB_1 \iclass \rD_2\rA_1\rD_2 \iclass \rA_1\rD_4 \ilevi \rE_7 \\
\rA_3\rA_1^2 & \rA_1\rA_1 \idiag \rA_1(\rA_1\rA_1) \ilevi \rA_1\rA_3 \ilevi \rE_7 \\
\rA_3\rA_2 & \rA_1\rT_1 = \rB_1\rD_1 \itens \rB_1\rD_3 \ilevi \rB_1\rD_3\rB_1 \iclass \rD_6 \ilevi \rE_7 \\
\rA_4\rA_1 & \rT_2 \\
\rD_4\rA_1 & \rC_2 \iauto \rA_3 \ilevi \rE_7 \\
\rD_4(a_1)\rA_1 & \rA_1\rA_1 \ilevi \rE_7 \\
(\rA_5)' & \rG_2 \iauto \rD_4 \ilevi \rE_7 \\
(\rA_5)'' & \rA_1\rA_1 \idiag \rA_1(\rA_1\rA_1\rA_1) \ilevi \rE_7 \\
\rD_5 & \rA_1\rA_1 \idiag \rA_1(\rA_1\rA_1) \ilevi \rE_7 \\
\rD_5(a_1) & \rA_1\rT_1 \ilevi \rE_7 \\
\rA_3\rA_2\rA_1 & \rA_1 \idiag \rA_1\rA_1 \ires \rA_4\rA_2 \ilevi \rE_7 \\
\rA_4\rA_2 & \rA_1 \idiag \rA_1\rA_1\rA_2 \ires \rA_1 \rA_2 \rA_3 \ilevi \rE_7 \\
\rA_5\rA_1 & \rA_1 \ilevi \rG_2 \iauto \rD_4 \ilevi \rE_7 \\
\end{array}
\]
\caption{Nilpotent centralizers in $\rE_7$}
\label{tab:e7-1}
\end{table}

\begin{table}
\[
\begin{array}{ll}
\text{Orbit} & (G^x_\red)^\circ \\
\hline
\rD_5\rA_1 & \rA_1 \idiag \rA_1\rA_1 \ilevi \rE_7 \\
\rD_5(a_1)\rA_1 & \rA_1 \ires \tilde{\rA}_2 \ilevi \rF_4 \iauto \rE_6 \ilevi \rE_7 \\
\rA_6 & \rA_1 \ilevi \rE_7 \\
\rD_6 & \rA_1 \idiag \rA_1\rA_1 \ilevi \rE_7 \\
\rD_6(a_1) & \rA_1 \ilevi \rE_7 \\
\rD_6(a_2) & \rA_1 \ilevi \rE_7 \\
\rE_6 & \rA_1 \idiag \rA_1\rA_1\rA_1 \ilevi \rE_7 \\
\rE_6(a_1) & \rT_1 \\
\rE_6(a_3) & \rA_1 \idiag \rA_1\rA_1\rA_1 \ilevi \rE_7
\end{array}
\]
\caption{Nilpotent centralizers in $\rE_7$, continued}
\label{tab:e7-2}
\end{table}

\begin{table}
\[
\begin{array}{ll}
\text{Orbit} & (G^x_\red)^\circ \\
\hline
\rA_1 & \rA_5 \ilevi \rE_6 \\
\rA_1^2 & \rB_3\rT_1 \iclass \rD_4\rT_1 \ilevi \rE_6 \\
\rA_2 & \rA_2\rA_2 \ilevi \rE_6 \\
\rA_1^3 & \rA_2\rA_1 \idiag (\rA_2\rA_2)\rA_1 \ilevi \rE_6 \\
\rA_2\rA_1 & \rA_2\rT_1 \ilevi \rE_6 \\
\rA_3 & \rB_2\rT_1 \iclass \rD_3 \rT_1 \ilevi \rE_6 \\
\rA_2\rA_1^2 & \rA_1\rT_1 = \rB_1\rT_1 \itens \rB_4\rT_1 \iclass \rD_5\rT_1 \ilevi \rE_6 \\
\rA_2^2 & \rG_2 \iauto \rD_4 \ilevi \rE_6 \\
\rA_3\rA_1 & \rA_1\rT_1 = \rC_1\rT_1 \itens \rD_2\rT_1 \ilevi \rE_6  \\
\rA_4 & \rA_1\rT_1 \ilevi \rE_6 \\
\rD_4 & \rA_2 \idiag \rA_2\rA_2 \ilevi \rE_6 \\
\rD_4(a_1) & \rT_2 \\
\rA_2^2\rA_1 & \rA_1 \idiag \rA_1\rA_1\rA_1 \ilevi \rE_6 \\
\rA_4\rA_1 & \rT_1 \\
\rA_5 & \rA_1 \ilevi \rE_6 \\
\rD_5 & \rT_1 \\
\rD_5(a_1) & \rT_1
\end{array}
\]
\caption{Nilpotent centralizers in $\rE_6$}
\label{tab:e6}
\end{table}

\begin{table}
\[
\begin{array}{ll}
\text{Orbit} & (G^x_\red)^\circ \\
\hline
\rA_1 & \rC_3 \ilevi \rF_4 \\
\tilde{\rA}_1 & \rA_3 \ilevi \rA_3\tilde{\rA}_1 \imax \rF_4 \\
\rA_1 \tilde{\rA}_1 & \rA_1\tilde{\rA}_1 \ires \rA_1\tilde{\rA}_2 \ilevi \rF_4 \\
\rA_2 & \tilde{\rA}_2 \ilevi \rF_4 \\
\tilde{\rA}_2 & \rG_2 \iauto \rD_4 \iclass \rB_4 \imax \rF_4 \\
\rB_2 & \rA_1\rA_1 \ilevi \rB_4 \imax \rF_4 \\
\rA_2 \tilde{\rA}_1 & \tilde{\rA}_1 \ilevi \rF_4 \\
\tilde{\rA}_2 \rA_1 & \rA_1 \ilevi \rG_2 \iauto \rD_4 \iclass \rB_4 \imax \rF_4  \\
\rB_3 & \rA_1 \ires \tilde{\rA}_2 \ilevi \rF_4 \\
\rC_3 & \rA_1 \ilevi \rF_4 \\
\rC_3(a_1) & \rA_1 \ilevi \rF_4
\end{array}
\]
\caption{Nilpotent centralizers in $\rF_4$}
\label{tab:f4}
\end{table}

\begin{table}
\[
\begin{array}{ll}
\text{Orbit} & (G^x_\red)^\circ \\
\hline
\rA_1 & \tilde{\rA}_1 \ilevi \rA_1\tilde{\rA}_1 \imax \rG_2 \\
\tilde{\rA}_1 & \rA_1 \ilevi \rA_1\tilde{\rA}_1 \imax \rG_2
\end{array}
\]
\caption{Nilpotent centralizers in $\rG_2$}
\label{tab:g2}
\end{table}

\subsection{Proof for \texorpdfstring{$\rE_8$}{E8}}\label{ss:E8}

When $x$ is distinguished, $(G^x_\red)^\circ$ is the trivial group; and when $x = 0$, $(G^x_\red)^\circ = G$.  For all remaining nilpotent orbits, we rely on the very detailed case-by-case descriptions of $(G^x_\red)^\circ$ given in~\cite[Chapter~15]{ls}.  In each case, that description shows that the embedding $(G^x_\red)^\circ \hookrightarrow G$ factors as a composition of various cases from Proposition~\ref{prop:pairlist}.

These factorizations are shown in Tables~\ref{tab:e8-1} and~\ref{tab:e8-2}. Here is a brief explanation of the notation used in these tables.  Nearly all groups mentioned are semisimple, and they are recorded in the tables by their root systems.  However, the notation ``$\rT_1$'' indicates a $1$-dimensional torus; this is used to indicate a reductive group with a $1$-dimensional center.  In a few cases, nonstandard names for root systems---such as $\rB_1$ or $\rC_1$, in place of $\rA_1$---are used when it is convenient to emphasize the role of a certain classical group.  The notation $\rD_1$ (meant to evoke $\SO_2$) is occasionally used as a synonym for $\rT_1$.

Finally, we remark that there are two orbits---labeled by $\rA_6$ and by $\rA_6\rA_1$---where the information given in~\cite{ls} is insufficient to finish the argument.  In each of these cases, $(G^x_\red)^\circ$ contains a copy of $\rA_1$ that is the centralizer of a certain copy of $\rG_2$ inside $\rE_7$, and~\cite{ls} does not give further details on this embedding $\rA_1 \hookrightarrow \rE_7$.  However, according to~\cite[\S 3.12]{seitz}, this $\rA_1$ is in fact (the derived subgroup of) a Levi subgroup of $\rE_7$. \qed

\subsection{Proof for \texorpdfstring{$\rE_7$}{E7} and \texorpdfstring{$\rE_6$}{E6}}

Recall that if $H \subseteq G$ is a closed subgroup, then there is a natural embedding $\cN_H \hookrightarrow \cN_G$
of nilpotent cones. We also recall that a subgroup $L \subseteq G$ is a Levi subgroup if and only if it 
is the centralizer of a torus $S \subset G$. 

\begin{lem}\label{lem:levi-cent}
Let $G$ be a reductive group, $S \subseteq G$ a torus with $L = C_G(S)$ a Levi subgroup, and suppose
 $x \in \cN_L \subseteq \cN_G$ is such that $S \subseteq (G^x_\red)^\circ$. 
Then $(L^x_\red)^\circ$ is a Levi subgroup of $(G^x_\red)^\circ$. 
\end{lem}
\begin{proof}
We clearly have $C_{G^x}(S) = G^x \cap L = L^x$.  Moreover, the identity component $(L^x)^\circ$ must be contained in $(G^x)^\circ \cap L = C_{(G^x)^\circ}(S)$.  But since the centralizer of a torus in a connected group is connected, we actually have
\[
(L^x)^\circ = C_{(G^x)^\circ}(S).
\]
Next, consider the semidirect product decomposition $(G^x)^\circ = (G^x_\red)^\circ \ltimes G^x_\unip$.  Let $g \in G^x$, and write it as $g = (g_r, g_u)$, with $g_r \in (G^x_\red)^\circ$, $g_u \in G^x_\unip$.  If $g$ centralizes $S$, then $g_r$ and $g_u$ must individually centralize $S$ as well.  In other words,
\begin{equation}\label{eqn:semidirect}
C_{(G^x)^\circ}(S) = C_{(G^x_\red)^\circ}(S) \ltimes C_{G^x_\unip}(S).
\end{equation}
Here, $C_{(G^x_\red)^\circ}(S)$ is a connected reductive group, and $C_{G^x_\unip}(S)$ is a normal unipotent group (which must be connected, because $C_{(G^x)^\circ}(S)$ is connected).  We conclude that~\eqref{eqn:semidirect} is a Levi decomposition of $(L^x)^\circ$.  In particular, we see that $(L^x_\red)^\circ = C_{(G^x_\red)^\circ}(S)$.
\end{proof}

We now let $G$ denote the simple, simply connected group of type $\rE_8$. If $G_0$ 
is the simple, simply connected group of type $\rE_7$ or $\rE_6$, then as explained in 
\cite[Lemma~11.14]{ls}, there is a simple subgroup $H$ of type $\rA_1$ or $\rA_2$ respectively,
such that $G_0 = C_{G}(H)$. 
Moreover, by \cite[16.1.2]{ls}, there exists a torus $S \subset H \subset G$ such that $G_0$ is the 
derived subgroup of the Levi subgroup $L = C_G(S)$.  Explicitly, let $\alpha_1, \ldots, \alpha_8$ be the simple roots for $G$, labelled as in~\cite{bourbaki}, and let $\alpha_0$ be the highest root.  The groups $H$ and $G_0$ can be described as follows.
\begin{equation}\label{eqn:e8-embed}
\begin{array}{c|c|c}
& \text{simple roots for $H$} & \text{simple roots for $L$ or $G_0$} \\
\hline
G_0 = \rE_7 & -\alpha_0 & \alpha_1, \ldots, \alpha_7 \\
G_0 = \rE_6 & \alpha_8, -\alpha_0 & \alpha_1, \ldots, \alpha_6
\end{array}
\end{equation}

Now it is explained in \cite[16.1.1]{ls}, that if $x \in \cN_{G_0} = \cN_L \subset \cN_G$, then the subgroup
$(G^x_{\red})^\circ$ must contain a conjugate of $H$. Without loss of generality we can assume that 
$x$ is chosen so that $H \subseteq (G^x_{\red})^\circ$. Hence, we can also assume that 
$S \subseteq (G^x_{\red})^\circ$. Thus, by Lemma~\ref{lem:levi-cent}, $(L^x_\red)^\circ$ 
is a Levi subgroup of $(G^x_\red)^\circ$ and we also have
\[
 (L^x_\red)^{\circ}\hspace{0mm}' \subseteq ((G_0^x)_\red)^\circ \subseteq (L^x_\red)^\circ.
\]
By Lemma~\ref{lem:derived}, Proposition~\ref{prop:pairlist}\eqref{it:levi} and \S\ref{ss:E8} we deduce that 
$(G,(G_0^x)_\red))^\circ)$ is a Donkin pair. 

Finally, to show that $(G_0, ((G_0^x)_\red)^\circ)$ is a Donkin pair, it will be sufficient to show that every
fundamental tilting module for $G_0$ is a summand of the restriction of a tilting module for $G$.  In more detail, let $\pi: \bX_G \twoheadrightarrow \bX_{G_0}$ be the map on weight lattices.   It is will known that if $\lambda$ is a dominant weight for $G$, then the $G_0$-tilting module $T_{G_0}(\pi(\lambda))$ occurs as a direct summand of $\Res^G_{G_0}(T_G(\lambda))$.  So it is enough to show that every fundamental weight for $G_0$ occurs as $\pi(\lambda)$ for some dominant $G$-weight $\lambda$.  Let $\varpi_1, \ldots, \varpi_8$ be the fundamental weights for $G$.  A short calculation with~\eqref{eqn:e8-embed} shows that $\pi(\varpi_1), \ldots, \pi(\varpi_7)$ are precisely the fundamental weights for $\rE_7$, and that $\pi(\varpi_1), \ldots, \pi(\varpi_6)$ are the fundamental weights for $\rE_6$. \qed

\begin{rmk}\label{rmk:e67-old}
One can also prove the theorem for $\rE_7$ and $\rE_6$ directly by writing down the embedding of each centralizer, as we did for $\rE_8$.  Here is a brief summary of how to carry out this approach.  Let $G$ be of type $\rE_8$, and let $G_0$ and $H$ be as in the discussion above.  As explained in~\cite[\S16.1]{ls}, we have
\[
((G_0^x)_\red)^\circ = C_{(G^x_\red)^\circ}(H).
\]
The computation of $C_{(G^x_\red)^\circ}(H)$ is explained in~\cite[\S16.1.4]{ls}, and the results are recorded in Tables~\ref{tab:e7-1}--\ref{tab:e6}, following the same notational conventions as in the $\rE_8$ case.
\end{rmk}

\subsection{Proof for \texorpdfstring{$\rF_4$}{F4}}

Let $G$ be the simple, simply connected group of type $\rE_8$.  Then~\cite[Lemma~11.7]{ls} implies that $G$ contains a simple subgroup $H$ of type $\rG_2$, and that its centralizer $G_0 = C_G(H)$ is a simple group of type $\rF_4$.  The embeddings of centralizers of nilpotent elements for $G_0 = \rF_4$ can then be computed using the method explained in Remark~\ref{rmk:e67-old}.  One caveat is that the name (i.e., the Bala--Carter label) of a nilpotent orbit usually changes when passing from $\rF_4$ to $\rE_8$.  The correspondence between these names can is given in~\cite[Proposition~16.10]{ls}.

We remark that in some cases, the book~\cite{ls} does not quite give enough details about embeddings of subgroups to establish our result, but in these cases, the relevant details can be found in~\cite[\S3.16]{seitz}.  Here is an example illustrating this.  The $\rF_4$-orbit labelled $\tilde\rA_1$ corresponds (by~\cite[Proposition~16.10]{ls}) to the $\rE_8$-orbit labelled $\rA_1^2$.  Let $x$ be an element of this orbit.  We have see that in $\rE_8$, $(G^x_\red)^\circ = \rB_6$, which embeds in the Levi subgroup $\rD_7 \subset \rE_8$.  The group $(G^x_\red)^\circ$ has a subgroup of type $\rD_3 \rB_3$, which embeds in $\rD_3 \rD_4 \subset \rD_7 \subset \rD_8$.  The explicit construction of $H = \rG_2$ in~\cite[\S3.16]{seitz} shows that it is contained in the second factor in each of $\rD_3 \rB_3 \subset \rD_3 \rD_4$.  It follows that $\rD_3 = \rA_3$ is contained in $((G^x_0)_\red)^\circ = C_{(G^x_\red)^\circ}(H)$, and then a dimension calculation shows that in fact $((G^x_0)_\red)^\circ = \rA_3$.

The results of these calculations are recorded in Table~\ref{tab:f4}.\qed

\subsection{Proof for \texorpdfstring{$\rG_2$}{G2}}

In this case, there are only two nilpotent orbits that are neither distinguished nor trivial.  From the classification, both of these orbits meet the maximal reductive subgroup $\rA_1\tilde\rA_1 \subset \rG_2$, and an argument explained in~\cite[\S16.1.4]{ls} shows that if $x$ belongs to either of these orbits, then the reductive part of its centralizer in $\rA_1\tilde\rA_1$ is equal to the reductive part of its centralizer in $\rG_2$.  See Table~\ref{tab:g2}. \qed


\end{document}